\documentclass[reqno]{amsart}

\usepackage{amssymb}  
\usepackage{enumitem}
\usepackage{hyperref}
\usepackage{bbm}

\newcommand*{\mailto}[1]{\href{mailto:#1}{\nolinkurl{#1}}}
\newcommand{\arxiv}[1]{\href{http://arxiv.org/abs/#1}{arXiv:#1}}

\newcommand{\msc}[1]{\href{http://www.ams.org/msc/msc2010.html?t=&s=#1}{#1}}

\newtheorem{theorem}{Theorem}[section]

\newtheorem{lemma}[theorem]{Lemma}

\newtheorem{corollary}[theorem]{Corollary}

\newtheorem{remark}[theorem]{Remark}

\newcommand{\id}{{\mathbbm 1}}
\newcommand{\R}{{\mathbb R}}

\newcommand{\Z}{{\mathbb Z}}
\newcommand{\C}{{\mathbb C}}

\newcommand{\Qr}{\mathsf{q}}
\newcommand{\PPr}{\mathsf{p}}

\newcommand{\Vr}{\mathsf{v}}
\newcommand{\Wr}{\mathsf{w}}

\newcommand{\I}{\mathrm{i}}

\newcommand{\tr}{\mathrm{tr}}

\newcommand{\supp}{\mathrm{supp}}
\newcommand{\loc}{\mathrm{loc}}

\newcommand{\cc}{\mathrm{c}}
\newcommand{\sign}{\mathrm{sgn}}

\newcommand{\be}{\begin{equation}}
\newcommand{\ee}{\end{equation}}

\newcommand{\ti}{\tilde}
\newcommand{\wt}{\widetilde}

\newcommand{\OO}{\mathcal{O}}

\newcommand{\cM}{\mathcal{M}}

\newcommand{\ledot}{\,\cdot\,}
\newcommand{\redot}{\cdot\,}
\newcommand{\NL}{(0-)}
\newcommand{\NLz}{(z,0-)}

\newcommand{\dip}{\upsilon}

\newcommand{\cL}{\mathcal{L}}

\renewcommand{\labelenumi}{(\roman{enumi})}
\renewcommand{\theenumi}{\roman{enumi}}
\numberwithin{equation}{section}

\begin{document}

\title[The Moment Problem and Indefinite Strings]{The Classical Moment Problem\\ and  Generalized Indefinite Strings}

\author[J.\ Eckhardt]{Jonathan Eckhardt}
\address{Faculty of Mathematics\\ University of Vienna\\ 
Oskar-Morgenstern-Platz 1\\ 1090 Wien\\ Austria}
\email{\mailto{Jonathan.Eckhardt@univie.ac.at}}
\urladdr{\url{http://www.mat.univie.ac.at/jonathan.eckhardt/}}

\author[A.\ Kostenko]{Aleksey Kostenko}
\address{Faculty of Mathematics and Physics\\ University of Ljubljana\\ Jadranska ul.\ 19\\ 1000 Ljubljana\\ Slovenia\\ and Faculty of Mathematics\\ University of Vienna\\ 
Oskar-Morgenstern-Platz 1\\ 1090 Wien\\ Austria\\ and RUDN University\\ Miklukho-Maklaya Str. 6\\ 
117198 Moscow\\Russia}
\email{\mailto{Aleksey.Kostenko@fmf.uni-lj.si};\ \mailto{Oleksiy.Kostenko@univie.ac.at}}
\urladdr{\url{http://www.mat.univie.ac.at/~kostenko/}}

\thanks{{\it Research supported by the Austrian Science Fund (FWF) under Grants No.\ P29299 (J.E.) and P28807 (A.K.) as well as by the Ministry of Education and Science of the Russian Federation under grant No.~02.A03.21.0008 (A.K.).}}


\begin{abstract}
We show that the classical Hamburger moment problem can be included in the spectral theory of generalized indefinite strings. Namely, we introduce the class of Krein--Langer strings and show that there is a bijective correspondence between moment sequences and this class of generalized indefinite strings. This result can be viewed as a complement to the classical results of M.\ G.\ Krein on the connection between the Stieltjes moment problem and Krein--Stieltjes strings and I.\ S.\ Kac on the connection between the Hamburger moment problem and $2\times 2$ canonical systems with Hamburger Hamiltonians.
\end{abstract}

\maketitle

\section{Introduction}

Let $\{s_k\}_{k\ge 0}$ be a sequence of real numbers. 
 The classical {\em Hamburger moment problem} is to find a positive Borel measure $\rho$ on $\R$ such that the numbers $s_k$ are its moments of order $k$, that is, such that 
\be\label{eq:Hamburger}
s_k=\int_{\R} \lambda^k\, \rho(d†\lambda),\quad k\ge 0.
\ee
 Every positive Borel measure $\rho$ on $\R$ that satisfies \eqref{eq:Hamburger}, is called a solution of the Hamburger moment problem with data $\{s_k\}_{k\ge 0}$. 
 Similarly, the {\em Stieltjes moment problem} is to find a positive Borel measure $\rho$ on $\R_{\ge 0}$ such that the numbers $s_k$ are its moments of order $k$, that is, such that 
\be\label{eq:Stieltjes}
s_k=\int_{\R_{\ge 0}} \lambda^k\, \rho(d\lambda),\quad k\ge0.
\ee
There are two principal questions: 
\begin{enumerate}[label=(\roman*), ref=(\roman*), leftmargin=*, widest=iii]
\item
For which sequences $\{s_k\}_{k\ge 0}$ are the moment problems solvable? 
\item
 Are solutions unique? If not, how to describe the set of all solutions?
\end{enumerate}
We are neither going to provide comprehensive historical details nor a complete discussion of solutions to both of these problems here. 
Instead, let us only refer to the book by N.\ I.\ Akhiezer \cite{Akh} (see also \cite{sim98}). 

It is widely known that the Hamburger moment problem is closely connected with the spectral theory of symmetric Jacobi (tri-diagonal) matrices. 
On the other side, it was discovered by M.\ G.\ Krein \cite{kr52} that the Stieltjes moment problem is closely connected with the  spectral theory of strings (now known as Krein strings), that is, spectral problems of the form
\be\label{eq:KSP}
 - f'' = z f \omega 
\ee
on an interval $[0,L)$, 
where $L\in (0,\infty]$ and $\omega$ is a positive Borel measure on $[0,L)$. 
The quantities $L$ and $\omega$ are usually referred to as the length and the mass density of the string, respectively. 
Both objects, Jacobi matrices and Krein strings, serve as certain canonical models for operators with simple spectra (for a nice account on canonical representations of self-adjoint operators we refer to a lecture by M.\ G.\ Krein \cite{krIII}). 
Another such model for operators with simple spectra is a $2\times 2$ canonical system \cite{ka83, wi14} and it was shown by I.\ S.\ Kac \cite{ka99a,ka99b} that the Hamburger moment problem can be included in the spectral theory of canonical systems with a special class of Hamiltonian functions termed {\em Hamburger Hamiltonians}. 

Motivated by the study of the indefinite moment problem, in \cite{krla79} (see also \cite{la76} and \cite{lawi98}) M.\ G.\ Krein and H.\ Langer introduced a new kind of spectral problem of the form
\be\label{eq:ISP}
 - f'' = z f \omega + z^2 f \dip
\ee
on an interval $[0,L)$, in which the spectral parameter enters in a nonlinear way. Here, the coefficient $\omega$ is a real-valued Borel measure on $[0,L)$ and $\dip$ is a positive Borel measure on $[0,L)$ supported on finitely many points. 
It turned out that spectral problems of the form \eqref{eq:ISP} also serve as a canonical model for operators with simple spectrum. 
More precisely, it was shown in \cite{IndefiniteString} that there is a one-to-one correspondence between spectral problems \eqref{eq:ISP} and canonical systems. 
In particular, this entails that every Herglotz--Nevanlinna function can be identified with the Weyl--Titchmarsh function of a unique spectral problem \eqref{eq:ISP}, however, for this, the assumptions on the coefficients have to be relaxed to allow $\omega$ to be a real-valued distribution in $H^{-1}_{\loc}([0,L))$ and $\dip$ to be a positive Borel measure on $[0,L)$. 
Similarly to Krein strings, we shall call such a triple $(L,\omega,\dip)$ a {\em generalized indefinite string}; see \cite{IndefiniteString}. 
In this respect, let us also mention briefly that a lot of the interest in spectral problems of the form~\eqref{eq:ISP} stems from the fact that they arise as isospectral problems for the conservative Camassa--Holm flow \cite{chlizh06, coiv08, ConservCH, CH2Real, ConservMP, CHString, hoiv11}. 

Our main aim here is to establish a connection between the Hamburger moment problem and the spectral theory of  generalized indefinite strings. More precisely, we will show that there is a one-to-one correspondence between moment sequences and a special class of generalized indefinite strings (which we decided to call {\em Krein--Langer strings}). This can be done in various ways. For example, one can use the results of I.\ S.\ Kac \cite{ka99a,ka99b} in conjunction with the correspondence between canonical systems and generalized indefinite strings (see Appendix \ref{sec:appA}). On the other hand, one can also prove this result by identifying moment sequences with (formal) continued fractions of the form  
\be\label{eq:StiCFmodinfinite}
  \cfrac{1}{-l_0\,z +  \cfrac{1}{\omega_0 + \dip_0\,z + \cfrac{1}{-l_1\,z + \cfrac{1}{\omega_1+\dip_1\,z+\cfrac{1}{\,\ddots\,}}}}}\,,
\ee 
which is the approach that we will follow here.
Notice that this kind of continued fraction is a slight modification of the one studied by T.-J.\ Stieltjes in \cite{sti} and subsequently applied to solve the Stieltjes moment problem. 
In fact, this generalization allows one to deal with the full (Hamburger) moment problem (as an alternative to employing the continued fractions used by H.\ Hamburger \cite{ha20}).
Continued fractions of the form~\eqref{eq:StiCFmodinfinite} naturally lead to spectral problems of the form~\eqref{eq:ISP} with coefficients $\omega$ and $\dip$ supported on a discrete set; Krein--Langer strings.  

A significant part of this article is of preliminary character. In Sections~\ref{sec:02} and~\ref{sec:03}, we first collect basic notions and  facts on Hamburger as well as Stieltjes moment problems and describe their respective relations to Jacobi matrices and Krein strings. Section~\ref{sec:04} then contains necessary information on canonical systems, Hamburger Hamiltonians and their connection with the Hamburger moment problem.  
After these preparations, we proceed to introduce the class of Krein--Langer strings in Section~\ref{sec:06} and subsequently prove our main result, which establishes a one-to-one correspondence between moment sequences and Krein--Langer strings.   

{\bf Notation.}  
 For any $a\in \R$, we set $\R_{>a}:= (a,\infty)$ and $\R_{\ge a}:=[a,\infty)$ as well as $\Z_{>a}:= \Z\cap \R_{>a}$ and $\Z_{\ge a}:= \Z\cap \R_{\ge a}$.   Moreover,  we will denote the canonical basis in $\ell^2(\Z_{\ge a})$ with $\{e_k\}_{k\ge a}$.

 If $I\subset \R$ is an interval, then we let $\cM(I)$ be the set of all real-valued Borel measures on $I$ and $\cM^+(I)$ the set of all positive Borel measures on $I$. In particular, we will use $\delta_x \in \cM^+(I)$ for the Dirac delta measure centered at $x\in I$. Finally, we will denote the characteristic function of a set $\Omega\subset \R$ with $\id_{\Omega}$.

\section{The Hamburger moment problem and Jacobi matrices}\label{sec:02}

The moment sequence $\{s_k\}_{k\ge 0}$  is called {\em positive} ({\em strictly positive}) if the Hankel determinants 
\be\label{eq:Delta0}
\Delta_{0,n} := \begin{vmatrix} s_0 & s_1 & \dots &s_{n} \\ s_1 & s_2 & \dots & s_{n+1} \\ \vdots & \vdots & \ddots & \vdots \\ s_{n} & s_{n+1} & \dots & s_{2n} \end{vmatrix}
\ee
 are non-negative (positive) for all $n\ge 0$. 

\begin{theorem}[H.\ Hamburger]\label{th:hamburger}
 There is a solution $\rho\in \cM^+(\R)$ to the Hamburger moment problem \eqref{eq:Hamburger} if and only if the sequence $\{s_k\}_{k\ge 0}$ is positive. 
 \end{theorem}

If a moment problem has a unique solution, then it is called {\em determinate}. Other\-wise it is called {\em indeterminate} and there are infinitely many solutions.

\begin{remark}\label{rem:positive}
 If the moment sequence $\{s_k\}_{k\ge 0}$ is positive but not strictly positive, then the Hamburger moment problem \eqref{eq:Hamburger} is determinate. 
 The unique solution $\rho$ is then supported on a finite set and one has $\Delta_{0,n}>0$ for all $n\in \{0,\dots, N-1\}$, where $N=\#\supp(\rho)$, and 
$\Delta_{0,n} = 0$ for all $n\ge N$.
\end{remark}

Suppose now that $\{s_{k}\}_{k\ge 0}$ is a strictly positive sequence and that $\rho$ is a solution to the Hamburger moment problem \eqref{eq:Hamburger}. Without loss of generality, we can assume that $s_0=1$, which means that $\rho$ is a probability measure. First, let us define the polynomials of the first kind associated with the measure $\rho$:
\begin{align}\label{eq:P_n}
P_0(z) & \equiv \frac{1}{\sqrt{s_0}}=1, & P_n(z) & := \frac{1}{\sqrt{\Delta_{0,n-1}\Delta_{0,n}}} 
\begin{vmatrix} 
s_0 & s_1 & \dots &s_{n} \\ 
s_1 & s_2 & \dots & s_{n+1} \\ 
\vdots & \vdots & \ddots & \vdots \\ 
s_{n-1} & s_n & \dots & s_{2n-1} \\
1 & z &  \dots & z^n \end{vmatrix},\quad n \geq 1.
\end{align}
Clearly, we have the asymptotics 
\be
P_n(z) = \sqrt{\frac{\Delta_{0,n-1}}{\Delta_{0,n}}} z^n - \frac{\Delta_{0,n-1}'}{\sqrt{\Delta_{0,n-1}\Delta_{0,n}}} z^{n-1} +\OO(z^{n-2}), \qquad z\rightarrow\infty. 
\ee
Here we set $\Delta_{0,-1}:=1$, $\Delta_{0,-1}':=0$, $ \Delta_{0,0}':=s_1$, and 
\be \label{eq:Delta0'}
\Delta_{0,n}':= \begin{vmatrix} 
s_0 & s_1 & \dots &s_{n-1} & s_{n+1} \\ 
s_1 & s_2 & \dots & s_{n} & s_{n+2} \\ 
\vdots & \vdots & \ddots & \vdots & \vdots\\ 
s_{n-1} & s_n & \dots & s_{2n-2} & s_{2n}\\
s_n & s_{n+1} &  \dots & s_{2n-1} & s_{2n+1} \end{vmatrix},\quad n\ge 1.
\ee 
The family $\{P_n\}_{n\ge 0}$ is orthonormal with respect to the measure $\rho$, however, it does not necessarily form a basis in $L^2(\R;\rho)$ (if the moment sequence is positive but not strictly positive, \eqref{eq:P_n} allows to define exactly $N$ polynomials $\{P_n\}_{n=0}^{N-1}$ and these polynomials serve as an orthonormal basis in $L^2(\R;\rho)$). Moreover, the polynomials $P_n$ satisfy the three-term recurrence relations
\be\label{eq:rec1}
b_{n-1} P_{n-1}(z) + a_n P_n(z) + b_{n}P_{n+1}(z) = z P_n(z),\quad n\ge 0,
\ee   
upon setting $P_{-1}\equiv 0$ for notational simplicity. Hereby, the coefficients in \eqref{eq:rec1} are given by $b_{-1}=0$ and 
\be\label{eq:b_n}
b_n=\int_{\R} \lambda P_n(\lambda) P_{n+1}(\lambda)\,\rho(d\lambda)= \frac{\sqrt{\Delta_{0,n-1}\Delta_{0,n+1}}}{\Delta_{0,n}},\quad n\ge 0, 
\ee
as well as by 
\be\label{eq:a_n}
 a_n = \int_{\R}\lambda P_n(\lambda)^2\,\rho(d\lambda)= \frac{\Delta_{0,n}'}{\Delta_{0,n}} - \frac{\Delta_{0,n-1}'}{\Delta_{0,n-1}},\quad n\ge 0.
\ee
The recurrence relations \eqref{eq:rec1} naturally generate the following Jacobi (tri-diagonal) matrix
\be\label{eq:Jacobi}
J:=J(a,b)=\begin{pmatrix} a_0 & b_0 & 0 & \dots \\ b_0 & a_1 & b_1 & \ddots \\ 0 & b_1 & a_2 & \ddots \\ \vdots & \ddots & \ddots & \ddots \end{pmatrix},
\ee
which defines a minimal (closed) symmetric operator in $\ell^2(\Z_{\ge 0})$. This operator is either self-adjoint or has deficiency indices $(1,1)$. In the former case, the matrix is said to be in the {\em limit point case} and in the latter it is said to be in the {\em limit circle case}. The next result is well known (see \cite{Akh} for example).

\begin{theorem}\label{th:HambCorresp}
The map 
\be\label{eq:PsiJ}
\Psi_{J}\colon \{s_k\}_{k\ge 0} \mapsto J(a,b),
\ee
 where $J(a,b)$ is the Jacobi matrix \eqref{eq:Jacobi} with coefficients defined by \eqref{eq:b_n} and \eqref{eq:a_n}, establishes a one-to-one correspondence between the set of strictly positive sequences with $s_0=1$ and the set of semi-infinite symmetric Jacobi matrices normalized by the condition $b_n>0$ for all $n\ge 0$. 
\end{theorem}

\begin{remark}\label{rem:scaling}
A few remarks are in order.
\begin{enumerate}[label=(\roman*), ref=(\roman*), leftmargin=*, widest=iii]
\item An analog of Theorem \ref{th:HambCorresp} for positive moment sequences which are not strictly positive also holds true. 
\item Clearly, if $\{s_k\}_{k\ge 0}$ is a positive moment sequence  and $c>0$, then the new sequence $\{\ti{s}_k\}_{k\ge 0}$ with $\ti{s}_k:=cs_k$ for all $k\ge 0$ is  positive as well. Moreover, it follows readily that $\ti{a}_n = a_n$ and $\ti{b}_n=b_n$ for all $n\ge 0$ in this case.
\end{enumerate}
\end{remark}

Let us next introduce the polynomials of the second kind:
\begin{align}\label{eq:Qn}
Q_0(z) & \equiv 0, 
 & Q_n(z) & :=\int_{\R} \frac{P_n(\lambda) - P_n(z)}{\lambda - z}\,\rho(d\lambda),\quad n\ge1.
\end{align}
Notice that the polynomials $Q_n$ do not actually depend on the choice of $\rho$ if the moment problem is indeterminate. Using \eqref{eq:P_n}, the polynomials $Q_n$ can be expressed via the moment sequence through 
\be\label{eq:QnMoments}
Q_n(z) =  \frac{1}{\sqrt{\Delta_{0,n-1}\Delta_{0,n}}} 
\begin{vmatrix} 
s_0 & s_1 & \dots &s_{n} \\ 
s_1 & s_2 & \dots & s_{n+1} \\ 
\vdots & \vdots & \ddots & \vdots \\ 
s_{n-1} & s_n & \dots & s_{2n-1} \\
R_{n,0}(z) & R_{n,1}(z) &  \dots & R_{n,n}(z)\end{vmatrix}, \quad n\ge 0, 
\ee
where $R_{n,0}\equiv 0$ and 
\be
R_{n,k}(z) = \sum_{m=0}^{k-1} s_{k-1-m}z^m
\ee
for all $k\in\{1,\dots,n\}$.  It is not difficult to check that
\be\label{eq:JacWronsk}
  P_n(z)Q_{n+1}(z) - P_{n+1}(z) Q_n(z) \equiv \frac{1}{b_n}, \quad n\ge 0. 
\ee

 One can characterize determinate Hamburger moment problems in terms of the Jacobi coefficients as well as the orthogonal polynomials (see \cite{Akh} for example).

\begin{theorem}\label{th:Akh}
Let $\{s_k\}_{k\ge 0}$ be a strictly positive sequence with $s_0=1$. Then the following conditions are equivalent:
\begin{enumerate}[label=(\roman*), ref=(\roman*), leftmargin=*, widest=iiii]
\item The Hamburger moment problem \eqref{eq:Hamburger} is determinate.
\item The Jacobi matrix $J(a,b)$ is in the limit point case.
\item The series $\sum_{n\ge 0} |P_n(0)|^2 + |Q_n(0)|^2$ diverges.
\item There exists $\lambda\in\R$ such that the series $\sum_{n\ge 0} |P_n(\lambda)|^2$ diverges.
\end{enumerate}
\end{theorem}

 Theorem \ref{th:HambCorresp} above establishes a connection between strictly positive sequences $\{s_k\}_{k\ge 0}$ with $s_0=1$ and semi-infinite Jacobi matrices. 
 This correspondence can also be described in another important way. Upon denoting 
\be\label{eq:JacobiN}
J_n=J_n(a,b) := 
\begin{pmatrix} a_0 & b_0 & 0 & \dots & 0\\ b_0 & a_1 & b_1 & \ddots &\vdots \\ 0 & b_1 & \ddots & \ddots & 0 \\ \vdots & \ddots & \ddots & a_{n-1} & b_{n-1}\\ 0 & \dots & 0 & b_{n-1} & a_n \end{pmatrix}
\ee
for all $n\ge 0$, let us define the function $m_n$ on $\C\backslash\R$ by 
\be\label{eq:m_nJacobi}  
m_n(z) := -\frac{Q_n(z)}{P_n(z)} = \big((J_{n-1} - z)^{-1}e_0,e_0\big), \quad z\in\C\backslash\R.
\ee
The rational function $m_n$ is a Herglotz--Nevanlinna function, that is, it is analytic, maps the upper complex half-plane into the closure of the upper complex half-plane and satisfies the symmetry relation
 \begin{align}\label{eqnHNsym}
  m_n(z)^\ast = m_n(z^\ast), \quad z\in\C\backslash\R. 
 \end{align}
Let us mention in this context that \eqref{eq:a_n} implies the following identity
\be
 \tr\, J_n = \sum_{k=0}^n a_k = \frac{\Delta_{0,n}'}{\Delta_{0,n}}.
\ee
Furthermore, this trace is equal to the sum of the roots of the polynomial $P_{n+1}$. 

The function $m_n$ admits the following asymptotic expansion (see \cite[\S I.4]{Akh}) 
\be\label{eq:assPade}
m_n(z) = -\frac{Q_n(z)}{P_n(z)} = -\frac{s_0}{z} - \frac{s_1}{z^2} - \dots  - \frac{s_{2n-1}}{z^{2n}} + \OO(z^{-2n-1}), \qquad z\rightarrow\infty.
\ee
Since the polynomials of the first and the second kind satisfy the recurrence relations in \eqref{eq:rec1}, $m_n$ admits the continued fraction expansion \cite[\S I.4.2]{Akh}: 
\be\label{eq:CFjacobi}
m_n(z) = \cfrac{1}{a_0 - z - \cfrac{b_0^2}{a_1 -z - \cfrac{b_1^2}{\,\ddots\,-\cfrac{b_{n-2}^2}{a_{n-1}-z}}}}\,, \quad z\in\C\backslash\R.
\ee 
For this reason, it makes sense to identify any semi-infinite Jacobi matrix \eqref{eq:Jacobi} with  
the formal continued fraction 
\be\label{eq:CFjacobiInf}
\cfrac{1}{a_0 - z - \cfrac{b_0^2}{a_1 -z - \cfrac{b_1^2}{a_2 - z - \cfrac{b_2^2}{\ddots}}}}\, .
\ee 
Clearly, its $n$-th order convergent is precisely the rational function $m_n$ having the asymptotic expansion \eqref{eq:assPade}. 

It remains to notice that the following limit 
\be\label{eq:WTjacobi}
m(z) := \lim_{n\to \infty}m_n(z) = \lim_{n\to \infty} -\frac{Q_n(z)}{P_n(z)},
\ee
exists for all $z\in\C\backslash\R$ if the Hamburger moment problem \eqref{eq:Hamburger} is determinate. The function $m$ is a Herglotz--Nevanlinna function and called the {\em Weyl--Titchmarsh function} of the Jacobi matrix \eqref{eq:Jacobi}.  
 Otherwise, when the moment problem is indeterminate, one can always find a convergent subsequence (see \cite[Chapter II.1]{Akh})
 \be
  m(z) := \lim_{k\to \infty}m_{n_k}(z) = \lim_{k\to \infty} -\frac{Q_{n_k}(z)}{P_{n_k}(z)}, \quad z\in\C\backslash\R.
 \ee
In any case, the limit function $m$ admits an integral representation of the form
\be\label{eq:WTjacobiIntRep}
m(z) = \int_{\R} \frac{\rho_0(d\lambda)}{\lambda-z},\quad z\in\C\backslash\R,
\ee
for some finite positive Borel measure $\rho_0$ on $\R$. 
The asymptotic expansion in~\eqref{eq:assPade} entails that the measure $\rho_0$ obtained in this way solves the Hamburger moment problem \eqref{eq:Hamburger}.

\section{The Stieltjes moment problem and Krein--Stieltjes strings}\label{sec:03}

 A (strictly) positive sequence $\{s_k\}_{k\ge 0}$ is called {\em (strictly) double positive}\footnote{This terminology does not seem to be standard, however, see \cite{gp}.} if the sequence $\{s_{k+1}\}_{k\ge 0}$ is (strictly) positive as well. 
 Clearly, the sequence $\{s_{k+1}\}_{k\ge 0}$ is (strictly) positive if and only if the determinants
\be\label{eq:Delta1}
\Delta_{1,n} := \begin{vmatrix} s_1 & s_2 & \dots &s_{n} \\ s_2 & s_3 & \dots & s_{n+1} \\ \vdots & \vdots & \ddots & \vdots \\ s_{n} & s_{n+1} & \dots & s_{2n-1} \end{vmatrix}
\ee
are non-negative  (positive) for all $n\ge 1$. In the following, we will also set $\Delta_{1,0}:=1$ for notational simplicity.

\begin{theorem}[T.-J.\ Stieltjes]\label{th:stieltjes}
 There is a solution $\rho\in \cM^+(\R_{\ge 0})$ to the Stieltjes moment problem \eqref{eq:Stieltjes} if and only if the sequence $\{s_k\}_{k\ge 0}$ is double positive.
\end{theorem}

If the moment sequence $\{s_k\}_{k\ge 0}$ is strictly double positive, then from the definition of the polynomials of the first kind in \eqref{eq:P_n} we have 
\begin{align}\label{eq:3.02}
P_n(0) & = (-1)^{n}\frac{\Delta_{1,n}}{\sqrt{\Delta_{0,n-1}\Delta_{0,n}}} \neq 0, & \sign(P_n(0)) & = (-1)^n,
\end{align}
for all $n\ge 0$. Hence, upon setting
\begin{align}
l_{n} & :=  |P_n(0)|^2, & \omega_n & := \frac{Q_n(0)}{P_{n}(0)} - \frac{Q_{n+1}(0)}{P_{n+1}(0)}= \frac{-1}{b_{n}P_{n}(0)P_{n+1}(0)},
\end{align}
for all $n\ge 0$  and using the recurrence relations \eqref{eq:rec1} with $z=0$, we conclude that  the coefficients of the Jacobi matrix \eqref{eq:Jacobi} admit the representation 
\begin{align}\label{eq:KrSti01}
a_n & =\frac{1}{l_{n}}\biggl(\frac{1}{\omega_{n-1}}+\frac{1}{\omega_{n}}\biggr), & b_n & =\frac{1}{\omega_{n}\sqrt{l_{n}l_{n+1}}},
\end{align}
for all $n\ge 0$, where we set $\frac{1}{\omega_{-1}}:=0$ for notational simplicity. Notice that by \eqref{eq:3.02} and \eqref{eq:b_n} we then have 
\begin{align}\label{eq:KrSti02}
l_{n} & = \frac{\Delta_{1,n}^2}{\Delta_{0,n-1}\Delta_{0,n}}, &  \omega_n & =\frac{\Delta_{0,n}^2}{\Delta_{1,n}\Delta_{1,n+1}},
\end{align}
for all $n\ge0$. 
Moreover, in this case the rational function $m_n$ defined in \eqref{eq:m_nJacobi} admits the Stieltjes continued fraction expansion
\be\label{eq:StiContdFrac}
m_{n}(z) = \cfrac{1}{-l_0\,z +  \cfrac{1}{\omega_0 + \cfrac{1}{\ddots +\cfrac{1}{-l_{n-1}\,z +\cfrac{1}{\omega_{n-1}}}}}}\, ,\quad z\in\C\backslash\R.
\ee 
Let us also mention the following formulas of Stieltjes \cite[(II.8), (II.11)]{sti}:
\begin{align}\label{eq:sumsStieltjes}
\sum_{k=0}^n l_k & = \sum_{k=0}^n |P_k(0)|^2 = \frac{\Delta_{2,n}}{\Delta_{0,n}}, & 
\sum_{k=0}^{n-1} \omega_k  & = -\frac{Q_n(0)}{P_n(0)} = - \frac{\Delta_{-1,n}}{\Delta_{1,n}},
\end{align}
where we defined the additional determinants
\begin{align}\label{eq:Delta-1}
\Delta_{-1,0} & = 0, & \Delta_{-1,n} & = \begin{vmatrix} 0 & s_0 & \dots &s_{n-1} \\ s_0 & s_1 & \dots & s_{n} \\ \vdots & \vdots & \ddots & \vdots \\ s_{n-1} & s_{n} & \dots & s_{2n-1} \end{vmatrix},\quad n\ge 1,
\end{align}
as well as  
\begin{align}\label{eq:Delta2}
\Delta_{2,0} & = 1, & \Delta_{2,n} & = \begin{vmatrix} s_2 & s_3 & \dots &s_{n+1} \\ s_3 & s_4 & \dots & s_{n+2} \\ \vdots & \vdots & \ddots & \vdots \\ s_{n+1} & s_{n+2} & \dots & s_{2n} \end{vmatrix},\quad n\ge 1.
\end{align} 

\begin{remark}\label{rem:scaling2}
If $\{s_k\}_{k\ge 0}$ is a double positive moment sequence and $c>0$, then so is the sequence $\{\ti{s}_k\}_{k\ge 0}$ defined by $\ti{s}_k:=cs_k$ for all $k\ge 0$. From \eqref{eq:KrSti02}, we then get
\begin{align}
\ti{\omega}_n & 
 = c\,\omega_n, & \ti{l}_n & 
  = c^{-1}\, l_n.
\end{align}
Notice that this transformation does not change the coefficients $a_n$ and $b_n$ in~\eqref{eq:KrSti01} (cf.\ Remark \ref{rem:scaling}), however, it does change the  continued fraction expansion~\eqref{eq:StiContdFrac}.
\end{remark}

It was observed by M.\ G.\ Krein \cite{kr52} (see also \cite[Appendix]{Akh}, \cite[\S 13]{kakr74}) that in the double positive case, the corresponding Jacobi matrix \eqref{eq:Jacobi} admits a mechanical interpretation. To this end, let us consider a string of length $L\in(0,\infty]$ carrying only point masses $\{\omega_n\}_{n=0}^{N-1}$  at the positions $\{x_n\}_{n=0}^{N-1}$ respectively, where we assume that $N\in\Z_{\ge0}\cup\{\infty\}$ and 
\begin{align}\label{eq:omegaKrSti1}
0 & =:x_{-1}<x_0 < x_1 < \dots <L, & l_n & := x_{n}-x_{n-1}>0.
\end{align}
In the case of infinitely many masses, we shall assume that $\lim_{n\to \infty}x_n=L$. Otherwise, if the string carries only finitely many point masses $N<\infty$, then we shall assume that $x_{N-1}<x_N:=L$, that is, there is no point mass at the right endpoint. 

If the ends of this string are fixed and it is stretched by a unit force, then small oscillations are described by the spectral problem associated with the corresponding Jacobi matrix $J$ given by \eqref{eq:Jacobi}, \eqref{eq:KrSti01}. Upon setting 
\be\label{eq:omegaKrSti2}
\omega:=\sum_{n=0}^{N} \omega_n\delta_{x_n} \in \cM^+([0,L)),
\ee
the corresponding difference equation \eqref{eq:rec1} can also be written as a spectral problem of the form (see \cite[\S 13]{kakr74})
\be\label{eq:KrString}
-f'' = z f\omega
\ee
on $[0,L)$. 
This differential equation has to be understood in a distributional sense and we postpone further details to Section \ref{sec:06}. Direct and inverse spectral theory for~\eqref{eq:KrString} with $\omega$ being an arbitrary positive Borel measure on $[0,L)$ has been developed by M.\ G.\ Krein in the 1950s (see \cite{kakr74, dymc76, kowa82}). Commonly, a pair $(L,\omega)$, where $L\in (0,\infty] $ and $\omega\in \cM^+([0,L))$ is called a {\em Krein string}. Such a Krein string $(L,\omega)$ is called {\em regular} if the length $L$ is finite and $\omega$ is a finite measure, that is,
\be\label{eq:KreinReg}
L+ \int_{[0,L)}\omega(dx) <\infty.
\ee
 Otherwise, the string is called {\em singular}. Strings of the particular form \eqref{eq:omegaKrSti1}, \eqref{eq:omegaKrSti2} are called {\em Stieltjes} or {\em Krein--Stieltjes strings}. 

With every Krein--Stieltjes string $(L,\omega)$ we can associate the following formal Stieltjes continued fraction
\be\label{eq:StiCFinfinite}
  \cfrac{1}{-l_0\,z +  \cfrac{1}{\omega_0 + \cfrac{1}{-l_1\,z + \cfrac{1}{\omega_1 + \cfrac{1}{\,\ddots\,}}}}}\,.
\ee 
If the number of point masses $N$ is finite, then \eqref{eq:StiCFinfinite} represents a rational Herglotz--Nevanlinna function.  
Otherwise, when there are infinitely many point masses, it was observed by Stieltjes that \eqref{eq:StiCFinfinite} converges for every $z\in \C\backslash\R_{\ge 0}$ if and only if at least one of the sums $\sum_{n\ge0}\omega_n= \omega([0,L))$ and $\sum_{n\ge0}l_n = L$ is infinite (or, equivalently, the Krein--Stieltjes string is singular). When the string $(L,\omega)$ is regular, the even order convergents $\Qr_{2n}(z)/\PPr_{2n}(z)$ and the odd order convergents $\Qr_{2n+1}(z)/\PPr_{2n+1}(z)$ of the continued fraction \eqref{eq:StiCFinfinite} still converge for every $z\in \C\backslash\R_{\ge 0}$, however, to different limits. Similar to the definition in Section~\ref{sec:02}, the limit 
\be\label{eq:mString}
m(z):= \lim_{n\to N} \frac{\Qr_{2n+1}(z)}{\PPr_{2n+1}(z)} =\lim_{n\to N}\cfrac{1}{-l_0\,z +  \cfrac{1}{\omega_0 + \cfrac{1}{\ddots +\cfrac{1}{\omega_{n-1} + \cfrac{1}{-l_n\,z} }}}}\, ,\quad z\in\C\backslash\R,
\ee    
will be called the {\em (principal) Weyl--Titchmarsh function} of the string $(L,\omega)$. Notice that it coincides with the {\em dynamical compliance} of the (dual) string (see \cite{kakr74, kawiwo07} for further details). 

\begin{theorem}[M.\ G.\ Krein]\label{th:StiCorresp} 
The map 
\be\label{eq:PsiS+}
\Psi_{\mathcal{S}}^+\colon \{s_k\}_{k\ge 0} \mapsto (L,\omega),
\ee
where $L\in (0,\infty]$ and $\omega\in \cM^+([0,L))$ are defined by \eqref{eq:KrSti02}, \eqref{eq:omegaKrSti1}, and  \eqref{eq:omegaKrSti2} establishes a one-to-one correspondence between the set of double positive sequences and the set of Krein--Stieltjes strings. 
\end{theorem}

\begin{remark}
A few remarks are in order.
\begin{enumerate}[label=(\roman*), ref=(\roman*), leftmargin=*, widest=iiii]
\item The Weyl--Titchmarsh function \eqref{eq:mString} of a Krein--Stieltjes string $(L,\omega)$ admits an integral representation of the form  
\be 
 m(z) = \int_{\R_{\ge0}} \frac{\rho_0(d\lambda)}{\lambda-z}, \quad z\in\C\backslash\R_{\ge 0},
\ee
for some finite positive Borel measure $\rho_0$ on $\R_{\ge0}$, which is a solution of the corresponding Stieltjes moment problem. 
\item The length $L$ of the string is related to the behaviour of $m$ near zero:
\be
 - \lim_{\varepsilon\downarrow0} \I\varepsilon\, m(\I\varepsilon) = \rho_0(\{0\}) = \frac{1}{L},
\ee
 where the fraction on the right-hand side has to be interpreted as zero when $L$ is infinite (cf. \cite[\S 11]{kakr74} and also \cite[\S 5]{IndefiniteString}).
 In particular, when zero is an isolated singularity of $m$, then it is a pole if and only if $L$ is finite. 
\item
 If the moment sequence $\{s_k\}_{k\ge 0}$ is double positive but not strictly double positive, then the Stieltjes moment problem \eqref{eq:Stieltjes} is determinate. The unique solution $\rho$ is then supported on a finite set and upon setting $N:=\#\supp(\rho)$ one has $\Delta_{1,n}>0$ for all $n\in \{1,\dots, N\}$ and $\Delta_{1,n}=0$ for all $n>N$ if $L$ is infinite and $\Delta_{1,n}>0$ for all $n\in \{1,\dots, N-1\}$ and $\Delta_{1,n} = 0$ for all $n\ge N$ 
 if $L$ is finite.  In this case, the equations \eqref{eq:KrSti02} define precisely $N$ points $\{x_n\}_{n=0}^{N-1}$ and $N$ weights $\{\omega_n\}_{n=0}^{N-1}$, that is, the corresponding measure $\omega$ is supported on a finite set.
\item  Notice that the Jacobi matrix \eqref{eq:Jacobi} can be written in the form \eqref{eq:KrSti01} only if the polynomials of the first kind do not vanish at $z=0$, or, equivalently 
\[
\Delta_{1,n} \neq 0 
\]
for all $n\ge 0$. In particular, the principal Weyl--Titchmarsh function admits an expansion \eqref{eq:StiCFinfinite} only if the above condition holds true; compare \cite{fw17}. 
 \end{enumerate}
\end{remark}

As in the previous section, one is again able to characterize determinate Stieltjes moment problems in terms of the corresponding Krein--Stieltjes strings.

\begin{theorem}\label{th:KreinStieltjes}
  Let $\{s_k\}_{k\ge 0}$ be a strictly double positive sequence. Then the following conditions are equivalent:
\begin{enumerate}[label=(\roman*), ref=(\roman*), leftmargin=*, widest=iiii]
\item The Stieltjes moment problem \eqref{eq:Stieltjes} is determinate.
\item The Krein--Stieltjes string $(L,\omega)$ is singular.
\item The series $\sum_{n\ge 0} l_n + \omega_n$ diverges.
\end{enumerate}
\end{theorem}

\begin{remark}
The equivalence $(i)\Leftrightarrow (iii)$ is due to Stieltjes \cite{sti} and the connection with strings together with the equivalence $(i)\Leftrightarrow (ii)$ was observed by M.\ G.\ Krein \cite{kr52} (see also \cite{kakr74}). 
\end{remark}

\section{The Hamburger moment problem and Hamburger Hamiltonians}\label{sec:04}

\subsection{Canonical systems} Let us first briefly review some facts about canonical systems as far as they are needed in this section; for more details we refer the reader to \cite{dB68, ka83, ro14, wi14}. 
 In order to set the stage, let $H$ be a locally integrable, real, symmetric and non-negative definite $2\times2$ matrix function on $[0,\infty)$. 
 Furthermore, we shall assume that $H$ is trace normed, that is, 
 \begin{align}\label{eq:II.03}
  \tr\, H(x) = H_{11}(x)+ H_{22}(x)=1 
 \end{align}
  for almost all $x\in[0,\infty)$, and also exclude the cases when 
 \begin{align}\label{eq:H0}
  H(x) = H_0:=\begin{pmatrix} 1 & 0\\ 0 & 0 \end{pmatrix}
 \end{align}
 for almost all $x\in[0,\infty)$. 
 A matrix function $H$ with all these properties is called a {\em Hamiltonian} and associated with such a function is the canonical first order system  
 \begin{align}\label{eq:II.01}
  \begin{pmatrix} 0 & 1\\ -1 & 0 \end{pmatrix} F' = z H F,
 \end{align} 
 with a complex spectral parameter  $z$. 
 We introduce the fundamental matrix solution $U$ of the canonical system~\eqref{eq:II.01} as the unique solution of the integral equation 
 \begin{align}\label{eqnIEcansysSF}
  U(z,x) = \begin{pmatrix} 1 & 0 \\ 0 & 1 \end{pmatrix} - z\int_0^x \begin{pmatrix} 0 & 1 \\ -1 & 0 \end{pmatrix} H(t) U(z,t)dt, \quad x\in[0,\infty),~z\in\C. 
 \end{align} 
 The {\em Weyl--Titchmarsh function} $m$ of the canonical system~\eqref{eq:II.01} is now defined by 
 \begin{align}\label{eqnWTmcansys}
  m(z)=  \lim_{x\rightarrow \infty} \frac{U_{11}(z,x)}{U_{12}(z,x)}, \quad z\in \C\backslash\R.
 \end{align}
 As a Herglotz--Nevanlinna function, it admits an integral representation of the form  
 \begin{align}\label{eq:Herglotz}
  m(z) = c_1 z + c_2 +  \int_\R \frac{1}{\lambda-z} - \frac{\lambda}{1+\lambda^2}\, \rho(d\lambda), \quad z\in\C\backslash\R, 
 \end{align}
 for some constants $c_1$, $c_2\in\R$ with $c_1\geq0$ and a positive Borel measure $\rho$ on $\R$ with  
 \begin{align}\label{eq:Herglotz2} 
  \int_{\R}\frac{\rho(d\lambda)}{1+\lambda^2} < \infty.
 \end{align}
 Note that the coefficient $c_1$ of the linear term can be read off the Hamiltonian $H$ immediately (see \cite[Lemma~2.5]{wi14}); 
\begin{align}
  c_1 = \sup\left\lbrace x\in[0,\infty) \,\left|\, H(t)=H_0 \text{ for almost all }t\in[0,x)\right.\right\rbrace. 
\end{align}
 
 It is a fundamental result of L.\ de Branges \cite{dB68} (see also \cite[Theorem~2.4]{wi14}) that indeed every Herglotz--Nevanlinna function arises as the Weyl--Titchmarsh function of a unique canonical system~\eqref{eq:II.01}.

\begin{theorem}[L.\ de Branges]\label{th:dB01}
 For every Herglotz--Nevanlinna function $m$ there is a Hamiltonian $H$ such that $m$ is the Weyl--Titchmarsh function of the canonical system \eqref{eq:II.01}. 
 Upon identifying Hamiltonians which coincide almost everywhere on $[0,\infty)$, this correspondence is also one-to-one.   
\end{theorem}

Let us also mention that in the case when 
$H(x) = H_0$ for almost all $ x\ge L$ with some $L\in (0,\infty)$, 
 straightforward calculations show that the corresponding Weyl--Titchmarsh function $m$ is given by
\be\label{eq:WTforCSfinite}
m(z) = \frac{U_{11}(z,L)}{U_{12}(z,L)},\quad z\in\C\backslash\R.
\ee
Hence, we can consider canonical systems \eqref{eq:II.01} on any finite interval $[0,L)$, the function \eqref{eq:WTforCSfinite} will be called the {\em principal Weyl--Titchmarsh function} and it coincides with the Weyl--Titchmarsh function of the canonical system whose Hamiltonian function is given by 
\be
\id_{[0,L)}(x)H(x) + \id_{[L,\infty)}(x) H_0,\quad x\in [0,\infty).
\ee
Hamiltonians on a finite interval are called {\em regular} and {\em singular} otherwise.

\subsection{Hamburger Hamiltonians}\label{ssecHH}
Following \cite{ka99a, ka99b}, let us now introduce a special class of Hamiltonians. 
To this end, fix some $L\in (0,\infty]$, an $N\in \Z_{\ge0} \cup \{\infty\}$ and let $\cL:=\{\ell_k\}_{k=0}^{N-1}$ and $\Theta:=\{\theta_k\}_{k=0}^{N}$ be real sequences such that $\theta_0=\frac{\pi}{2}$ and 
\begin{align}\label{eq:normal1}
 \ell_k & > 0, & \theta_k & <\theta_{k+1}<\theta_k+\pi,
\end{align}
for all $k\in \{0,\dots,N-1\} $.  We then set 
\begin{align}\label{eq:xlHamil}
x_{-1} & :=0; & x_{k} & :=x_{k-1}+\ell_k,\quad k\in \{0,\dots,N-1\}.
\end{align}
We also assume that $x_{N-1}<x_N:=L$ (so that $l_N:=x_N-x_{N-1}\in(0,\infty]$) and $\theta_N\not\in\pi\Z$ if $N<\infty$, and 
\be
L:=\lim_{k\to \infty} x_k = \sum_{k\ge 0} \ell_k,
\ee
in the case $N=\infty$. Next define the Hamiltonian function $H_{\cL,\Theta}\colon [0,L)\to \R^{2\times 2}$ by 
\be\label{eq:Hambhamil}
H_{\cL,\Theta}(x) :=  \sum_{k=0}^NH_{\theta_k} \id_{[x_{k-1},x_{k})}(x),\quad x\in [0,L), 
\ee
where the matrix $H_\theta$ is defined by 
\be\label{eq:Htheta}
 H_\theta:= \begin{pmatrix} \cos^2(\theta) & \cos(\theta)\sin(\theta) \\ \cos(\theta)\sin(\theta)  & \sin^2(\theta) \end{pmatrix}.
\ee
 Hamiltonians of the above form are called {\em Hamburger Hamiltonians} \cite{ka99a, ka99b}. 
 Notice that the requirement \eqref{eq:normal1} implies that every interval $(x_{k-1},x_k)$ is maximal $H$-indivisible of type $\theta_k$. 

Before we formulate the main results from \cite{ka99a, ka99b}, we need the following well-known fact. For every $n\in\{0,\dots,N\}$, denote by $H_{\cL,\Theta}^n$ the Hamiltonian defined on the interval $[x_{n-1},L)$ by 
\be
H_{\cL,\Theta}^n(x):=\sum_{k= n}^N H_{\theta_k} \id_{[x_{k-1},x_{k})}(x),\quad x\in [x_{n-1},L).
\ee
If $U^n$ is the fundamental matrix solution of the system
\be
  \begin{pmatrix} 0 & 1\\ -1 & 0 \end{pmatrix}F' = z H_{\cL,\Theta}^n F
\ee
on $[x_{n-1},L)$, then the corresponding Weyl--Titchmarsh function $\wt{m}_n$ is defined by 
\be
\wt{m}_n(z):= \lim_{x\to L}\frac{U_{11}^n(z,x)}{U_{12}^n(z,x)}, \quad z\in\C\backslash\R.
\ee

\begin{lemma}\label{lem:recursHn}
For every $n\in\{0,\dots,N-1\}$ one has 
\be
\wt{m}_n(z) = \ell_{n} z + \wt{m}_{n+1}(z), \quad z\in\C\backslash\R,
\ee
if $\theta_{n} \in \pi\Z$, and 
\be
\wt{m}_n(z) = \cot(\theta_{n}) + \cfrac{1}{-\ell_n \sin^2(\theta_{n}) z  + \cfrac{1}{-\cot(\theta_n) + \wt{m}_{n+1}(z)}}, \quad z\in\C\backslash\R,
\ee
whenever $\theta_{n}\notin\pi \Z$.
\end{lemma}

\begin{proof}
Noting that the fundamental matrix-solution $U^n$ is given by
\[
\begin{split}
U^n(z,x) 
&=\begin{pmatrix} 1- z (x-x_{n-1}) \cos(\theta_n)\sin(\theta_n) &  - z (x-x_{n-1}) \sin^2(\theta_n) \\   z (x-x_{n-1}) \cos^2(\theta_n) & 1+   z (x-x_{n-1}) \cos(\theta_n)\sin(\theta_n)\end{pmatrix},
\end{split}
\]
for all $x\in [x_{n-1},x_n]$ and $z\in\C$, we get
\[
U^n(z,x) = U^{n+1}(z,x)U^n(z,x_n),\quad x\in [x_n,L),~z\in\C.
\]
Hence straightforward calculations show that
\[
\wt{m}_n(z) = \frac{\wt{m}_{n+1}(z) (1-z\ell_n\cos(\theta_n)\sin(\theta_n)) + z\ell_n\cos^2(\theta_n)}{-\wt{m}_{n+1}(z) z\ell_n\sin^2(\theta_n) + 1+z\ell_n\cos(\theta_n)\sin(\theta_n)}, \quad z\in\C\backslash\R,
\]
which readily establishes the claim.
\end{proof}

We define $\kappa$ as the number of nonzero (modulo $\pi$) elements of $\{\theta_k\}_{k=1}^N$. For every $j\in\{0,\ldots,\kappa\}$, let us denote by $k(j)$ the largest integer $k\in\{0,\ldots,N\}$ such that the number of nonzero (modulo $\pi$) elements of $\theta_0,\theta_1,\dots,\theta_{k-1}$ is exactly $j$. 
This definition is so that the sequence $\{\theta_{k(j)}\}_{j=0}^\kappa$ enumerates all nonzero (modulo $\pi$) members of the sequence $\{\theta_k\}_{k=0}^N$. 
In particular, we have $k(0)=0$ since $\theta_0=\pi/2$. 
Moreover, let us set 
\begin{align}\label{eq:l_nk}
l_j & := \ell_{k(j)}\sin^2(\theta_{k(j)}), & \omega_j & := \cot(\theta_{k(j+1)}) - \cot(\theta_{k(j)}),
\end{align}
as well as
\be\label{eq:dip_nk} 
\dip_j:= \begin{cases} 0, & k(j+1)-k(j)=1, \\ \ell_{k(j)+1}, & k(j+1)-k(j)=2.\end{cases}
\ee
In view of \eqref{eq:normal1}, the coefficients $\dip_j$ are well-defined for all $j\in\{0,\ldots,\kappa-1\}$.

\begin{corollary}\label{cor:HfinCFexpansion}
If  $N$ is finite,  then the Weyl--Titchmarsh function $m$ admits the continued fraction expansion 
\be\label{eq:mHfinCF}
m(z) = \cfrac{1}{-l_0\, z + \cfrac{1}{\omega_0 + \dip_0\, z + \cfrac{1}{\,\ddots\,  + \cfrac{1}{\omega_{\kappa-1} + \dip_{\kappa-1}\, z + \cfrac{1}{-l_{\kappa}\, z}}}}}\,, \quad z\in\C\backslash\R.
\ee
\end{corollary}

\begin{proof}
It suffices to note that $m$ coincides with $\wt{m}_0$ and
\[
\wt{m}_N(z) = \cot(\theta_N) - \frac{1}{\ell_N \sin^2(\theta_N) z} = \cot(\theta_{n(\kappa)}) + \frac{1}{- l_\kappa z}, \quad z\in\C\backslash\R,
\]
 since by normalization $\theta_N\not\in\pi\Z$, and then apply Lemma \ref{lem:recursHn}.
\end{proof}

In particular, the Weyl--Titchmarsh function corresponding to a Hamburger Hamiltonian with finite $N$ is a rational function that vanishes at $\infty$. The converse holds true as well.

\begin{corollary}\label{cor:IPfinHamb}
Every rational Herglotz--Nevanlinna function that vanishes at $\infty$ is the Weyl--Titchmarsh function of a canonical system with a Hamburger Hamiltonian with finite $N$.
\end{corollary}

\begin{proof}
 Taking into account \eqref{eq:l_nk}, \eqref{eq:dip_nk} and \eqref{eq:mHfinCF}, it suffices to show that every rational Herglotz--Nevanlinna function $m$ that vanishes at $\infty$ admits an expansion of the form in \eqref{eq:mHfinCF}. 
 The proof of this fact is constructive and it can be seen as an analog of the Euclidian algorithm for Herglotz--Nevanlinna functions. Indeed, the assumptions on $m$ imply that $m = \Qr/\PPr$, where the polynomials $\Qr$ and $\PPr$ do not have common zeros and $\deg(\PPr) = \deg(\Qr) +1 =: n$. Unless $m$ is identically zero, we may write 
\[
m(z) = \frac{1}{-m_0(z)}, \quad z\in\C\backslash\R,
\]
where $m_0 = -\PPr/\Qr$ is a rational Herglotz--Nevanlinna function. Since $\deg(\PPr) = \deg(\Qr) +1$, there is an $l_0>0$ such that $m_0(z) = l_0\,z + \wt{m}_0(z)$, where $\wt{m}_0$ is again a rational Herglotz--Nevanlinna function satisfying $\wt{m}_0=\wt{\PPr}/\Qr$, where $n-2\le \deg(\wt{\PPr}) \le \deg(\Qr) = n-1$. Hence, unless $\wt{m}_0$ is identically zero,  
\[
m(z) = \cfrac{1}{-l_0\, z + \cfrac{1}{m_1(z)}}, \quad z\in\C\backslash\R, 
\]
where $m_1:= -1/\wt{m}_0$ is a rational Herglotz--Nevanlinna function. Therefore, there are constants $\omega_0\in\R$ and $\dip_0\in\R_{\ge 0}$ such that $m_1(z) = \omega_0  + \dip_0\,z + \wt{m}_1(z)$, where $\wt{m}_1$ is a rational Herglotz--Nevanlinna function that vanishes at $\infty$. Since $\wt{m}_0$ is bounded near $\infty$, at least one of the coefficients $\omega_0$ or $\dip_0$ is non-zero. Moreover, we have $\wt{m}_1 = \wt{\Qr}/\wt{\PPr}$, where the polynomials $\wt{\Qr}$ and $\wt{\PPr}$ do not have common zeros and $\deg(\wt{\PPr}) = \deg(\wt{\Qr}) +1 \le n-1$.  Upon applying the same procedure to $\wt{m}_1$, we arrive at the representation \eqref{eq:mHfinCF} after finitely many iterations.
\end{proof}

\subsection{Connection with the moment problem} 
Notice that if $\theta_k\notin \pi\Z$ for all $k\in\{0,\ldots,N\}$, then it follows from~\eqref{eqnWTmcansys} and Corollary~\ref{cor:HfinCFexpansion} that the Weyl--Titchmarsh function $m$ admits the Stieltjes continued fraction expansion 
\be
 m(z) = \lim_{n\rightarrow N} \cfrac{1}{-l_0\, z + \cfrac{1}{\omega_0 + \cfrac{1}{\,\ddots\,  + \cfrac{1}{\omega_{n-1} + \cfrac{1}{-l_{n}\, z}}}}}\,, \quad z\in\C\backslash\R,
\ee
where the coefficients are given by 
\begin{align}
l_k & = \ell_k\sin^2(\theta_k),  & \omega_k & = \cot(\theta_{k+1})-\cot(\theta_{k}).
\end{align}
 This establishes a connection between canonical systems with such Hamburger Hamiltonians and continued fractions of the form~\eqref{eq:CFjacobiInf}, and thus also with Jacobi matrices. Indeed, taking \eqref{eq:KrSti01} into account, the corresponding Jacobi coefficients (after some calculations) are given by
\begin{align}\label{eq:Jac=CS}
a_n & = -\frac{\cot(\theta_{n+1}-\theta_n) + \cot(\theta_{n}-\theta_{n-1})}{\ell_n}, & b_n & = \frac{1}{\sin(\theta_{n+1}-\theta_n)\sqrt{\ell_{n+1}\ell_n}},
\end{align}
 where $\theta_{-1}:=0$ for notational simplicity. Moreover, the second formula in~\eqref{eq:sumsStieltjes} together with the first formula in~\eqref{eq:KrSti02} imply 
\begin{align}\label{eq:CSmoments}
\cot(\theta_{n}) & = - \frac{\Delta_{-1,n}}{\Delta_{1,n}}, & \ell_n & =  \frac{\Delta_{-1,n}^2+\Delta_{1,n}^2}{\Delta_{0,n-1}\Delta_{0,n}}.
\end{align}
 It was observed by I.\ S.\ Kac in \cite{ka99a,ka99b} that in fact \eqref{eq:Jac=CS} establishes a one-to-one correspondence between  Hamburger Hamiltonians with $N=\infty$ and $\ell_0=1$, semi-infinite Jacobi matrices and thus also strictly positive sequences $\{s_k\}_{k\ge0}$ with $s_0=1$. More precisely, to this end we only need to set 
\be\label{eq:CSmoments2}
\theta_n:= 0\ ({\rm mod}\;\pi)
\ee  
  in \eqref{eq:CSmoments} if $\Delta_{1,n}=0$. Note that  the lengths $\ell_n$ are indeed positive for all $n\ge 0$ since we have the inequality   
\be\label{eq:Deltas}
   \Delta_{1,n} \Delta_{-1,n+1} - \Delta_{1,n+1} \Delta_{-1,n} \not= 0, 
\ee
which follows upon evaluating~\eqref{eq:JacWronsk} at zero and using~\eqref{eq:P_n} as well as~\eqref{eq:QnMoments} to compute the values of $P_n$ and $Q_n$ at zero. 

\begin{theorem}[I.\ S.\ Kac]\label{th:kac}
The map
\be\label{eq:psiH}
\Psi_H\colon \{s_k\}_{k\ge 0}\mapsto H_{\cL,\Theta},
\ee
where $H_{\cL,\Theta}$ is the Hamburger Hamiltonian \eqref{eq:Hambhamil} defined by \eqref{eq:CSmoments}, \eqref{eq:CSmoments2} establishes a one-to-one correspondence between the set of positive sequences and the set of Hamburger Hamiltonians. 
\end{theorem}

Combining this result with Theorem \ref{th:HambCorresp}, we obtain a one-to-one correspondence between Hamburger Hamiltonians with infinite $N$ as well as $\ell_0=1$ and semi-infinite symmetric Jacobi matrices (cf. \eqref{eq:Jac=CS}). Notice that if $\{P_n\}_{n\ge 0}$ and $\{Q_n\}_{n\ge 0}$ are the corresponding orthogonal polynomials defined in \eqref{eq:P_n} and \eqref{eq:Qn}, then the formulas in \eqref{eq:CSmoments} read 
\begin{align}\label{eq:thetal=PQ1}
\cot(\theta_{n}) & = -\frac{Q_n(0)}{P_n(0)}, & \ell_n & = |P_n(0)|^2 + |Q_n(0)|^2,
\end{align}
and conversely 
\begin{align}\label{eq:thetal=PQ2}
P_n(0) & = \sqrt{\ell_n}\sin(\theta_n), & Q_n(0) & = - \sqrt{\ell_n}\cos(\theta_n).
\end{align}

\begin{remark}\label{rem:IPsystems}
Let $\{s_k\}_{k\ge 0}$ be a positive sequence and $H_{\cL,\Theta}$ the corresponding Hamburger Hamiltonian. Let us emphasize that the associated Weyl--Titchmarsh function \eqref{eq:WTforCSfinite} admits an integral representation of the form  
\be 
 m(z) = \int_{\R} \frac{\rho_0(d\lambda)}{\lambda-z}, \quad z\in\C\backslash\R,
\ee
for some finite positive Borel measure $\rho_0$ on $\R$, which is a solution of the corresponding Hamburger moment problem. 
\end{remark}

Again, we are able to characterize determinate Hamburger moment problems in terms of the corresponding Hamiltonian (see \cite{ka99a,ka99b}).

\begin{theorem}[I.\ S.\ Kac]\label{th:kac2}
  Let $\{s_k\}_{k\ge 0}$ be a strictly positive sequence. Then the following conditions are equivalent:
\begin{enumerate}[label=(\roman*), ref=(\roman*), leftmargin=*, widest=iiii]
\item The Hamburger moment problem \eqref{eq:Hamburger} is determinate.
\item The Hamburger Hamiltonian $H_{\cL,\Theta}$ is singular.
\item The series $\sum_{n\ge 0} \ell_n$ diverges.
\end{enumerate}
\end{theorem}

\section{The Hamburger moment problem and Krein--Langer strings}\label{sec:06}

\subsection{Generalized indefinite strings} Let us first briefly review some facts about generalized indefinite strings; for more details we refer the reader to \cite{IndefiniteString, CHPencil, CHString}. 
 To this end, fix some $L\in(0,\infty]$, let $\omega\in H^{-1}_{\loc}([0,L))$ be a real-valued distribution on $[0,L)$ and $\dip$ be a positive Borel measure on $[0,L)$.  
 We will first discuss the meaning of the differential equation 
 \begin{align}\label{eqnDEho}
  -f''  = z  f \omega + z^2 f \dip, 
 \end{align}
 where $z$ is a complex spectral parameter. 
 Of course, this equation has to be understood in a distributional sense:   
  A solution of~\eqref{eqnDEho} is a function $f\in H^1_{\loc}([0,L))$ such that 
 \begin{align}
  \Delta_f h(0) + \int_{0}^L f'(x) h'(x) dx = z\, \omega(fh) + z^2 \int_{[0,L)} f(t)h(t) \,\dip(dt), \quad h\in H^1_{\cc}([0,L)),
 \end{align}
 for some constant $\Delta_f\in\C$. 
 In this case, the constant $\Delta_f$ is uniquely determined and will henceforth always be denoted with $f'\NL$ for apparent reasons.  
 Of course, there are also several other ways of introducing the same notion of solutions. 
 Upon choosing particular test functions $h_x\in H^1_{\cc}([0,L))$ given by 
 \begin{align}
  h_x(t) = \begin{cases} x-t, & t\in[0,x), \\ 0, & t\in[x,L), \end{cases}
 \end{align}
 for every $x\in[0,L)$, one observes that a function $f\in H^1_{\loc}([0,L))$ is a solution of~\eqref{eqnDEho} if and only if one has 
 \begin{align}\label{eqnDEinhoInt}
   f(x) & = f(0) + f'\NL x - z\,\omega(fh_x) - z^2 \int_{[0,L)} f(t)h_x(t)\, \dip(dt), \quad x\in[0,L). 
 \end{align}
  Note that this formulation simply reduces to the usual integral equation (as used in, for example, \cite[\S1]{kakr74}, \cite[Section~1]{la76}, see also \cite{CHPencil}) if $\omega$ is a Borel measure:
  \be\label{eq:IntEqn}
  f(x) = f(0) + f'(0-)x - z\int_{[0,x)} (x-t) f(t)\,\omega(dt) - z^2 \int_{[0,x)} (x-t) f(t)\,\dip(dt).
  \ee
 
For every $z\in\C$, we introduce the fundamental system of solutions $c(z,\redot)$, $s(z,\redot)$ of the differential equation~\eqref{eqnDEho} satisfying the initial conditions
 \begin{align}\label{eq:csat0}
  c(z,0)& = s'\NLz =1, &  c'\NLz & = s(z,0) =0.
 \end{align}
  This allows us to define the Weyl--Titchmarsh function $m$ by
  \begin{align}\label{eq:mFalt}
    m(z) = \lim_{x\rightarrow L} -\frac{c(z,x)}{zs(z,x)}, \quad z\in\C\backslash\R.
  \end{align}
 As a Herglotz--Nevanlinna function (see \cite{IndefiniteString}), the function $m$ has an integral representation of the form  \eqref{eq:Herglotz}--\eqref{eq:Herglotz2} again.

Similarly to Krein strings, a triple $(L,\omega,\dip)$ such that $L\in(0,\infty]$, $\omega$ is a real-valued distribution in $H^{-1}_{\loc}([0,L))$ and $\dip$ is a positive Borel measure on $[0,L)$ is called a {\em generalized indefinite string}. Such a string $(L,\omega,\dip)$ is called {\em regular} if the length $L$ is finite, $\omega\in H^{-1}([0,L))$ and $\dip([0,L))<\infty$, that is, if 
\be\label{eq:defSing}
L+ \int_0^L \Wr(x)^2 dx + \int_{[0,L)} \dip(dx) < \infty, 
\ee
where $\Wr\in L^2_{\mathrm{loc}}([0,L))$ is the anti-derivative of $\omega$ specified by  
\be\label{eq:defAnti}
  \omega(h) = - \int_0^L \Wr(x) h'(x)dx, \quad h\in H^1_{\mathrm{c}}([0,L)).
\ee
Otherwise, the string is called {\em singular}.
Note that although the class of generalized indefinite strings contains the class of Krein strings, the notion of regularity does not coincide on this subset. However, the regularity of Krein strings corresponds to the indeterminacy of the Stieltjes moment problem whereas the regularity of generalized indefinite strings correlates with the indeterminacy of the Hamburger moment problem (see Theorem~\ref{th:determinacy} below).

 \begin{theorem}[\cite{IndefiniteString}]\label{thmIP}
  For every Herglotz--Nevanlinna function $m$ there is a unique generalized indefinite string $(L,\omega,\dip)$ which has $m$ as its Weyl--Titchmarsh function. 
\end{theorem}

\subsection{Krein--Langer strings}\label{ss:5.2}
Let us now introduce a special class of generalized indefinite strings carrying only point masses $\omega_j$ and dipoles $\dip_j$ located at points $x_j$ which can only accumulate at $L$. More precisely, let $L\in (0,\infty]$, $\kappa\in \Z_{\ge 0}\cup\{\infty\}$, $\{x_j\}_{j=0}^{\kappa-1}$ be a sequence of reals such that
\begin{align}\label{eq:omegaKrSti1b}
0 & =:x_{-1}<x_0 < x_1 < \dots <L, 
\end{align}
$\{\omega_j\}_{j=0}^{\kappa-1}$ be a sequence of reals and $\{\dip_j\}_{j=0}^{\kappa-1}$ be a sequence of non-negative reals. 
For definiteness, we shall assume that
\be
|\omega_j| + \dip_j>0
\ee
for all $j\in \{0,\dots,\kappa-1\}$. 
In the case of infinitely many masses and dipoles, we shall assume that $L=\lim_{j\to \infty}x_j$. If the string has only finitely many masses and dipoles, then we shall assume that $x_{\kappa-1}<x_\kappa:=L$, that is, there is neither a point mass nor a dipole at the right end. 
Next,  we define 
\be
 l_j  := x_{j}-x_{j-1}
 \ee
for all $j\in\{0,\ldots,\kappa\}$. Finally,  we set
\begin{align}\label{eq:omegadipKL}
\omega & :=\sum_{j=0}^{\kappa-1} \omega_j\delta_{x_j} \in \cM([0,L)), & \dip &:=\sum_{j=0}^{\kappa-1} \dip_j\delta_{x_j} \in \cM^+([0,L)).
\end{align}
A generalized indefinite string $(L,\omega,\dip)$ of the above form \eqref{eq:omegaKrSti1b}--\eqref{eq:omegadipKL} will be called a {\em Krein--Langer string} (due to its first appearance in the work of M.\ G.\ Krein and H.\ Langer \cite{krla79}; see also \cite{la76}). 

 Let us now consider the corresponding spectral problem \eqref{eqnDEho}. Since $\omega$ and $\dip$ are both measures, the differential equation~\eqref{eqnDEho}  reduces to the integral equation~\eqref{eq:IntEqn}, which is nothing but  
\be\label{eqLintEqnKL}
f(x) = f(0) + f'(0-)x - \sum_{x_j<x} (x-x_j) (z\,\omega_j+z^2\dip_j)\, f(x_j),\quad x\in [0,L).
\ee
Clearly, the solution $f$ is thus continuous and piece-wise linear. Moreover, evaluating $f$ at the points $x_j$, we get
\begin{align} \label{eq:recKL1}
 \begin{split}
f'(x_j+) - f'(x_{j}-) & = - (z\,\omega_j + z^2\dip_j) f(x_j), \\
f(x_{j}) - f(x_{j-1}) & = l_{j} f'(x_{j-1}+) = l_{j} f'(x_{j}-).
 \end{split}
\end{align}
 In particular, the representation~\eqref{eqLintEqnKL} shows that the functions $c(\ledot,x)$ and $s(\ledot,x)$ are polynomials for every $x\in[0,L)$. 
 Upon setting
\be
\wt{m}_j(z):= -\frac{c(z,x_j)}{z\, s(z,x_j)},\quad z\in\C\backslash\R,
\ee
it is not difficult to see using \eqref{eq:recKL1} that $\wt{m}_j$ admits the following continued fraction expansion
\be\label{eq:cfracKL}
\wt{m}_j(z) = \cfrac{1}{-l_0\,z +  \cfrac{1}{\omega_0 + \dip_0\, z + \cfrac{1}{\,\ddots\, + \cfrac{1}{\omega_{j-1}+ \dip_{j-1}\, z + \cfrac{1}{-l_j\, z}}}}}\, ,\quad z\in\C\backslash\R,
\ee
for all $j\in \{0,\dots,\kappa\}$. Comparing this representation with the discussion in Subsection~\ref{ssecHH} suggests that there is a one-to-one correspondence between canonical systems with Hamburger Hamiltonians and Krein--Langer strings. The next result is due to M.\ G.\ Krein and H.\ Langer \cite{krla79}. 

\begin{theorem}[M.\ G.\ Krein--H.\ Langer]\label{th:KL}
Let $m$ be a rational Herglotz--Nevanlinna function that vanishes at $\infty$. 
Then there exists a unique Krein--Langer string $(L,\omega,\dip)$ with only finitely many masses and dipoles such that the corresponding Weyl--Titchmarsh function coincides with $m$. 
\end{theorem}

\begin{proof}
The proof of Corollary \ref{cor:IPfinHamb} shows that the function $m$ has a representation of the form~\eqref{eq:mHfinCF}, which allows to construct the desired Krein--Langer string. 
\end{proof}

\subsection{Connection with the moment problem}
Theorem~\ref{th:KL} shows that there is a one-to-one correspondence between the set of positive sequences which are not strictly positive and the set of Krein--Langer strings having only finitely many point masses and dipoles. 
In fact, the coefficients of the Krein--Langer string corresponding to such a sequence are given by~\eqref{eq:l_nk}--\eqref{eq:dip_nk} and~\eqref{eq:CSmoments}. 
Our next aim is to establish the full analog of Theorem~\ref{th:kac} for Krein--Langer strings. To this end, let us suppose that $\{s_k\}_{k\ge 0}$ is a strictly positive sequence. 
For every $j\ge0$, we define $k(j)$ as the largest integer $k\ge0$ such that the number of nonzero elements of $\Delta_{1,0},\Delta_{1,1},\dots,\Delta_{1,k-1}$ is exactly $j$. 
According to this definition, the sequence $\{\Delta_{1,k(j)}\}_{j=0}^\infty$ enumerates all nonzero members of the sequence $\{\Delta_{1,k}\}_{k=0}^\infty$. 
In particular, note that we have $k(0)=0$ since $\Delta_{1,0}=1$.  
Now let us define 
\begin{align}\label{eq:lnkStr}
l_j & := \frac{\Delta^2_{1,k(j)}}{\Delta_{0,k(j)-1}\Delta_{0,k(j)}}, & x_j & := \sum_{i=0}^j l_i, & L & := \sum_{i\ge0} l_i, 
\end{align}
for every $j\ge0$ 
as well as  
\begin{align}\label{eq:lnkStr2}
\omega_j & := \frac{\Delta^2_{0,k(j)}}{\Delta_{1,k(j)}\Delta_{1,k(j)+1}}, & \dip_j & := 0,
\end{align}
if $k(j+1)-k(j)=1$, and 
\begin{align}\label{eq:lnkStr3}
\omega_j & :=\frac{\Delta_{-1,k(j)}}{\Delta_{1,k(j)}} - \frac{\Delta_{-1,k(j+1)}}{\Delta_{1,k(j+1)}}, & 
\dip_j & := \frac{\Delta^2_{-1,k(j)+1}}{\Delta_{0,k(j)}\Delta_{0,k(j)+1}},  
\end{align}
if $k(j+1)-k(j)=2$. 
It follows from~\eqref{eq:Deltas} that there are no consecutive zeros within the sequence $\{\Delta_{1,k}\}_{k=0}^\infty$, which ensures that the above quantities are well-defined.  

\begin{theorem}\label{th:main}
The map
\be
\Psi_{\mathcal{S}}\colon \{s_k\}_{k\ge 0}\mapsto (L,\omega,\dip)
\ee
where $L\in (0,\infty]$, $\omega\in \cM([0,L))$ and $\dip\in\cM^+([0,L))$ are defined by \eqref{eq:omegaKrSti1b}--\eqref{eq:omegadipKL} and \eqref{eq:lnkStr}--\eqref{eq:lnkStr3} establishes a one-to-one correspondence between the set of positive sequences and the set of Krein--Langer strings. 
\end{theorem}

\begin{proof}
In view of Theorem \ref{th:KL}, it suffices to prove  the claim only for strictly positive sequences. 
Moreover, we just need to show that the map is surjective.
 However, this follows from the continued fraction expansion \eqref{eq:cfracKL}. 
 Indeed, sending $j$ to infinity there, we see that every Krein--Langer string can be identified with a formal infinite continued fraction of this type. It remains to use the one-to-one correspondence between continued fractions of this type and strictly positive sequences as well as noting that the coefficients therein are related via~\eqref{eq:lnkStr}--\eqref{eq:lnkStr3}.
\end{proof}

\begin{remark}
One can also prove Theorem~\ref{th:main} by combining Kac's Theorem \ref{th:kac} with the transformation connecting canonical systems with generalized indefinite strings; see \cite[Section~6]{IndefiniteString} and Appendix~\ref{sec:appA}.
\end{remark}

\begin{corollary}\label{cor:KLstrOPL}
Let $\{s_k\}_{k\ge 0}$ be a strictly positive sequence  and $(L,\omega,\dip)$ the corresponding Krein--Langer string. If $\{P_n\}$ and $\{Q_n\}$ are the corresponding orthogonal polynomials, then
\begin{align}\label{eq:lnkStr4}
l_j & = |P_{k(j)}(0)|^2, & 
\omega_j & =  \frac{Q_{k(j)}(0)}{P_{k(j)}(0)} - \frac{Q_{k(j+1)}(0)}{P_{k(j+1)}(0)}, 
\end{align}
and 
\be\label{eq:lnkStr5}
\dip_j = \begin{cases} 0, & k(j+1) - k(j)=1, \\ |Q_{k(j)+1}(0)|^2, & k(j+1) - k(j)=2. \end{cases}
\ee
\end{corollary}

\begin{proof}
This follows readily upon comparing \eqref{eq:lnkStr}--\eqref{eq:lnkStr3} with 
\begin{align}\label{eq:PQzero}
 P_n(0) & = (-1)^n \frac{\Delta_{1,n}}{\sqrt{\Delta_{0,n-1}\Delta_{0,n}}}, & Q_n(0) & = (-1)^n \frac{\Delta_{-1,n}}{\sqrt{\Delta_{0,n-1}\Delta_{0,n}}},
\end{align}
also employing the relation 
\[
\Delta_{1,n+1}\Delta_{-1,n} - \Delta_{1,n}\Delta_{-1,n+1} = \Delta_{0,n}^2, \quad n\ge0,
\]
which follows from evaluating~\eqref{eq:JacWronsk}  at zero and using~\eqref{eq:b_n}.
\end{proof}

The next result extends Stieltjes' formulas \eqref{eq:sumsStieltjes} to the case of positive sequences and generalized indefinite strings.
 
\begin{corollary}\label{cor:StiExtended}
If $\{s_k\}_{k\ge 0}$ is a strictly positive sequence and $(L,\omega,\dip)$ is the corresponding Krein--Langer string, then 
\begin{align}\label{eq:tr01}
x_j = \sum_{i=0}^j l_i & = \frac{\Delta_{2,k(j)}}{\Delta_{0,k(j)}},\\
\omega([0,x_j)) =\sum_{i=0}^{j-1}\omega_i & = -\frac{\Delta_{-1,k(j)}}{\Delta_{1,k(j)}},\label{eq:tr02} \\
 \int_0^{x_j} \Wr(x)^2 dx + \sum_{i=0}^{j-1} \dip_i & = -\frac{\Delta_{-2,k(j)}}{\Delta_{0,k(j)}},  \label{eq:tr03}
\end{align}
where we defined the additional determinants 
\begin{align}\label{eq:Delta-2}
 \Delta_{-2,0} & := 0, & \Delta_{-2,n} & := \begin{vmatrix} 0 & 0 & s_0 & \dots &s_{n-1} \\ 0& s_0 & s_1 & \dots & s_{n}  \\  s_0 & s_1 & s_2 & \dots & s_{n+1}\\  \vdots & \vdots & \vdots & \ddots & \vdots \\ s_{n-1} & s_{n} & s_{n+1} & \dots & s_{2n} \end{vmatrix},\quad n\ge 1.
\end{align}
\end{corollary}

\begin{proof}
Let $\{P_n\}$ and $\{Q_n\}$ be the orthogonal polynomials associated with the strictly positive sequence $\{s_k\}_{k\ge 0}$. Since 
\be\label{eq:sumP}
x_j = \sum_{i=0}^j l_i = \sum_{i=0}^j |P_{k(i)}(0)|^2 = \sum_{n=0}^{k(j)} |P_{n}(0)|^2 = \frac{\Delta_{2,k(j)}}{\Delta_{0,k(j)}}, 
\ee
we arrive at the first equality. The second equality follows from 
\[
    \omega([0,x_j)) =\sum_{i=0}^{j-1}\omega_i  =  - \frac{Q_{k(j)}(0)}{P_{k(j)}(0)} = -\frac{\Delta_{-1,k(j)}}{\Delta_{1,k(j)}},
\] 
where we used~\eqref{eq:lnkStr4} and~\eqref{eq:PQzero}.
For the last equality, on the one hand we have 
\[ 
  \int_0^{x_j} \Wr(x)^2 dx = \sum_{i=0}^{j} l_{i} \omega([0,x_i))^2 = \sum_{i=0}^{j} |Q_{k(i)}(0)|^2
\]
and on the other side also 
\[
\dip([0,x_j)) =  \sum_{i=0}^{j-1} \dip_i = \mathop{\sum_{i=0}^{j-1}}_{k(i+1)-k(i)=2} |Q_{k(i)+1}(0)|^2,
\]
which adds up to
\be\label{eq:sumQ}
  \int_0^{x_j} \Wr(x)^2 dx + \sum_{i=0}^{j-1} \dip_i = \sum_{n=0}^{k(j)} |Q_{n}(0)|^2
\ee
and thus we obtain~\eqref{eq:tr03} in view of \cite[Theorem~A.6]{sim98}. 
\end{proof}

Finally, we may again characterize determinate Hamburger moment problems in terms of the corresponding generalized indefinite string.

\begin{theorem}\label{th:determinacy}
  Let $\{s_k\}_{k\ge 0}$ be a strictly positive sequence. Then the following conditions are equivalent:
\begin{enumerate}[label=(\roman*), ref=(\roman*), leftmargin=*, widest=iiii]
\item The Hamburger moment problem \eqref{eq:Hamburger} is determinate.
\item The generalized indefinite string $(L,\omega,\dip)$ is singular.
\item The series below diverges; 
 \begin{align}
 \sum_{n\ge 0} l_{n+1} + l_{n+1} \Wr_{n}^2 + \dip_{n} & = \infty, & \Wr_n & :=\sum_{k=0}^n \omega_k.
 \end{align}
\end{enumerate}
\end{theorem}
 
 \begin{proof}
The equivalence $(ii)\Leftrightarrow (iii)$ is immediate from the definition (see \eqref{eq:defSing}) and the following equality for Krein--Langer strings
\[
L+ \int_0^L \Wr(x)^2 dx + \int_{[0,L)} \dip(dx) = l_0 + \sum_{n\ge 0} l_{n+1} + l_{n+1} \Wr_n^2 + \dip_n.
\]
 Hence it suffices to show that $(i)\Leftrightarrow (iii)$. By Theorem \ref{th:Akh}, the Hamburger moment problem is determinate if and only if the series $\sum_{n\ge 0} |P_n(0)|^2 + |Q_n(0)|^2$ diverges. It remains to use \eqref{eq:sumP} and \eqref{eq:sumQ}.
 \end{proof}
 
Note that the last condition in the previous result reduces to a criterion by Hamburger \cite[Theorem~0.5]{Akh} when the moment sequence $\{s_k\}_{k\ge 0}$ is strictly double positive.

\appendix

\section{Hamburger Hamiltonians and Krein--Langer strings}\label{sec:appA}

It has been established in \cite[Section~6]{IndefiniteString} that Hamiltonians $H$ are in one-to-one correspondence  with generalized indefinite strings $(L,\omega,\dip)$. 
Our aim here is to show that this transformation also gives a correspondence between Hamburger Hamiltonians and Krein--Langer strings. 
In conjunction with the results of \cite{ka99a, ka99b}, these facts can be used for an alternative proof of Theorem~\ref{th:main}.

\subsection{From Krein--Langer strings to Hamburger Hamiltonians}
To begin with, let $(L_S,\omega,\dip)$ be a Krein--Langer string as in Section \ref{ss:5.2}. We want to find the corresponding Hamiltonian $H$. Following \cite[Section~6]{IndefiniteString}, we first define the two functions $\Wr\colon [0,L_S)\to \R$ and $\Vr\colon [0,L_S)\to \R_{\ge0}$ by 
 \begin{align}\label{eq:WrVr}
\Wr(x) & :=\omega([0,x)) = \sum_{x_k<x}\omega_k, &  \Vr(x) & :=\dip([0,x)) = \sum_{x_k<x}\dip_k,
\end{align}
for all $x\in [0,L_S)$. Observe that one has 
\begin{align}\label{eq:WrVr_k}
\Wr(x) & = \Wr_n= \sum_{k=0}^{n} \omega_k, & \Vr(x) & = \Vr_n: = \sum_{k=0}^n\dip_k,
\end{align}
for all $x\in(x_{n},x_{n+1}]$. Next, we introduce the function $\varsigma:[0,L_S]\rightarrow[0,\infty]$ by 
  \begin{align}\label{eqnIPsigmatrans}
   \varsigma(x) = x +\Vr(x) + \int_0^x \Wr(t)^2 dt, \quad x\in[0,L_S], 
  \end{align}
  as well as its generalized inverse $\xi$ on $[0,\infty)$ via 
  \begin{align}\label{eqnIPxitrans}
   \xi(s) = \sup\left\lbrace x\in[0,L_S)\,|\, \varsigma(x)\leq s\right\rbrace, \quad s\in[0,\infty). 
  \end{align}
 Let us point out explicitly that $\xi(s)=L_S$ for $s\in[\varsigma(L_S),\infty)$ provided that $\varsigma(L_S)$ is finite.
  On the other side, if $\varsigma(L_S)$ is not finite, then $\xi(s)<L_S$ for every $s\in[0,\infty)$ but $\xi(s)\rightarrow L_S$ as $s\rightarrow\infty$.
  
  Clearly, the function $\varsigma$ is positive and strictly increasing. Moreover, it is continuous on $[0,L_S]$ except for possibly the set of points $\{x_k\}_{k=0}^{\kappa-1}$, where we have
 \be
 \varsigma(x_k+) - \varsigma(x_k) = \dip_k, \quad k\in\{0,\ldots,\kappa-1\}.
 \ee
 Furthermore, we have $\varsigma(x)=x$ on $[0,x_0]$ and 
\be
\varsigma(x) = x + \Vr_n + \sum_{k=0}^{n-1} l_{k+1} \Wr_k^2 + \Wr_n^2 (x-x_n),\quad x\in (x_n,x_{n+1}],~n\in\{0,\ldots,\kappa-1\}.
\ee
This readily entails that 
\begin{align}
\varsigma(x_n+) & = x_n + \Vr_n + \sum_{k=0}^{n-1} l_{k+1} \Wr_k^2, &
\varsigma(x_{n+1}) & = x_{n+1} + \Vr_n + \sum_{k=0}^{n} l_{k+1} \Wr_k^2,
\end{align}
for all $n\in\{0,\ldots,\kappa-1\}$. 
We now compute that 
\be
\xi(s) = \begin{cases} s, & s\in (0,x_0),\\ x_k, & s\in  (\varsigma(x_k),\varsigma(x_k+)), \\ x_k + \frac{s-\varsigma(x_k+)}{1+\Wr_k^{2}}, & s\in (\varsigma(x_k+),\varsigma(x_{k+1})), \\ L_S, & s\in(\varsigma(L_S),\infty),\end{cases}
\ee
and as a consequence, we see that 
\be
\xi'(s) = \begin{cases} 1, & s\in (0,x_0),\\ 0, & s\in  (\varsigma(x_k),\varsigma(x_k+)), \\ (1+\Wr_k^{2})^{-1}, & s\in (\varsigma(x_k+),\varsigma(x_{k+1})), \\ 0, & s\in(\varsigma(L_S),\infty).\end{cases}
\ee
Then the Hamiltonian $H$ can be given in terms of $\xi$ and $\Wr$ by  (see \cite[(6.10)]{IndefiniteString})
  \begin{align}\label{eqnHamequ}
   H(s) := \begin{pmatrix}1 - \xi'(s) & \xi'(s)\Wr(\xi(s)) \\ \xi'(s)\Wr(\xi(s)) & \xi'(s) \end{pmatrix}, \quad s\in[0,\infty).
  \end{align}
 With the considerations above, it is now not difficult to conclude that
 \be
 H(s) = \begin{cases} H_{\pi/2}, & s\in (0,x_0), \\ H_0, & s\in (\varsigma(x_k),\varsigma(x_k+)), \\ H_{\theta_k}, & s\in (\varsigma(x_k+),\varsigma(x_{k+1})), \\ H_0, & s\in(\varsigma(L_S),\infty), \end{cases}
 \ee 
  where $H_\theta$ is given by \eqref{eq:Htheta} and the $\theta_k$ are such that 
  \be
  \cot(\theta_k) = \Wr_k = \sum_{j=0}^{k} \omega_j. 
  \ee
   In conclusion, we have seen that the Hamiltonian $H$ corresponding to the Krein--Langer string $(L_S,\omega,\dip)$ is a Hamburger Hamiltonian with length $\varsigma(L_S)$.

\subsection{From Hamburger Hamiltonians to Krein--Langer strings} 
For the converse direction, let $H_{\cL,\Theta}$ be a Hamburger Hamiltonian \eqref{eq:Hambhamil} with coefficients satisfying \eqref{eq:normal1} as well as \eqref{eq:xlHamil} and denote with $(L_S,\omega,\dip)$ the corresponding generalized indefinite string. 
We first note that 
\be
H_{22}(x) = \sum_{k= 0}^N \id_{[x_{k-1},x_{k})}(x) \sin^2(\theta_k),\quad x\in [0,L).
\ee
and introduce the locally absolutely continuous, non-decreasing function $\xi$ by
  \begin{align}\label{eq:xiHamb01}
  \begin{split}
   \xi(x) &=\sum_{k= 0}^N \sin^2(\theta_k) \int_0^x \id_{[x_{k-1},x_{k})}(s)ds \\
   &=\sum_{x_{k-1}<x} \sin^2(\theta_k) \int_0^x \id_{[x_{k-1},x_{k})}(s)ds, \quad x\in [0,\infty).
   \end{split}
  \end{align} 
In particular, this reduces to 
\be\label{eq:xiHamb}
 \xi(x) = \sin^2(\theta_k)(x-x_{k-1}) + \sum_{j=0}^{k-1} \ell_j\sin^2(\theta_j),\quad x\in [x_{k-1},x_{k}),
\ee  
for all $k\in\{0,\ldots,N\}$ and we observe that
 \be
 \xi(x_{k-1}) = \sum_{j=0}^{k-1} \ell_j\sin^2(\theta_j),\quad k\in\{0,\ldots,N\}.
 \ee 
 It follows that the length of the generalized indefinite string $(L_s,\omega,\dip)$ is given by 
 \be
  L_S = \lim_{x\rightarrow\infty} \xi(x) = \sum_{k=0}^N \ell_k\sin^2(\theta_k). 
 \ee
The left-continuous generalized inverse $\varsigma\colon [0,L_S]\to [0,\infty]$ of $\xi$ satisfies 
\be
\varsigma(x) =  x_{k-1} + \frac{x-\xi(x_{k-1})}{\sin^2(\theta_k)} \in (x_{k-1},x_k)
\ee
for all $x\in (\xi(x_{k-1}), \xi(x_{k}))$ if $\sin(\theta_k)\neq 0$, and 
\begin{align}
 x_k = \varsigma(\xi(x_k)+) = \varsigma(\xi(x_k)) + \ell_k  = x_{k-1} + \ell_k
\end{align}
if $\sin(\theta_k)=0$. From this it follows that 
\begin{align}
H_{22}(\varsigma(x)) & = \sum_{j= 0}^N \id_{[x_{j-1},x_{j})}(\varsigma(x)) \sin^2(\theta_j) = \sin^2(\theta_k), \\
H_{12}(\varsigma(x)) & = \sum_{j= 0}^N \id_{[x_{j-1},x_{j})}(\varsigma(x)) \cos(\theta_j)\sin(\theta_j) = \cos(\theta_k)\sin(\theta_k),
\end{align}
if $x\in (\xi(x_{k-1}), \xi(x_{k}))$ for some $k\in\{0,\ldots,N\}$ with $\sin(\theta_k)\neq 0$. 
Thus, the normalized anti-derivative $\Wr$ of the distribution $\omega$ is given by 
\be
\Wr(x) = \frac{H_{12}(\varsigma(x))}{H_{22}(\varsigma(x))}   = \cot(\theta_k) 
\ee
when $x\in (\xi(x_{k-1}), \xi(x_{k}))$ for some $k\in\{0,\ldots,N\}$ with $\sin(\theta_k)\neq 0$. 
As a consequence, the distribution $\omega$ turns out to be a real-valued Borel measure on $[0,L_S)$: 
\begin{align}
\omega & = \mathop{\sum_{k=0}^{N-1}}_{\sin(\theta_k)\not=0}  \omega_k \delta_{\xi(x_k)}, & \omega_k & = \begin{cases} \cot(\theta_{k+1})-\cot(\theta_k), & \theta_{k+1}\not\in\pi\Z, \\ \cot(\theta_{k+2})-\cot(\theta_{k}), & \theta_{k+1}\in\pi\Z. \end{cases}
\end{align}
The positive Borel measure $\dip$ on $[0,L_S)$ is given via its distribution function by 
  \begin{align}
   \int_{[0,x)} \dip(ds) = \varsigma(x) - x - \int_0^x \Wr(t)^2dt, \quad x\in[0,L_S). 
  \end{align}
  If $x\in (\xi(x_{k-1}), \xi(x_{k}))$ for some $k\in\{0,\ldots,N\}$ with $\sin(\theta_k)\neq 0$, then we compute
  \be
    \int_{[0,x)} \dip(ds) = x_{k-1} - \xi(x_{k-1})   - \int_0^{\xi(x_{k-1})} \Wr(t)^2dt,
  \ee 
  from which we finally conclude that 
  \be
    \dip = \mathop{\sum_{k=0}^{N-1}}_{\sin(\theta_k)=0}  \ell_k \delta_{\xi(x_k)}. 
  \ee
  Thus, the generalized indefinite string $(L_S,\omega,\dip)$ corresponding to the Hamburger Hamiltonian $H_{\cL,\Theta}$ is a Krein--Langer string.

\end{document}